\documentclass[11pt,a4paper]{amsart}

\RequirePackage[numbers]{natbib}
\RequirePackage[colorlinks,citecolor=blue,urlcolor=blue]{hyperref}
\usepackage{graphicx}
\usepackage{enumitem}
\usepackage[utf8]{inputenc}
\usepackage[english]{babel}
\usepackage{latexsym}
\usepackage{amssymb}
\usepackage{amsmath}
\usepackage{dsfont}

\newtheorem{thm}{Theorem}
\newtheorem{lemma}[thm]{Lemma}
\newtheorem{proposition}[thm]{Proposition}

\newtheorem{definition}[thm]{Definition}
\newtheorem{remark}[thm]{Remark}

\newcommand{\C}{\mathbb{C}}

\renewcommand{\H}{\mathbb{H}}

\renewcommand{\Im}{\mathrm{Im} \,}
\renewcommand{\epsilon}{\varepsilon}

\def\:{\mathop{:}}

\DeclareFontEncoding{LS1}{}{}
\DeclareFontSubstitution{LS1}{stix}{m}{n}
\DeclareSymbolFont{symbols2}{LS1}{stixfrak}{m}{n}
\DeclareMathSymbol{\typecolon}{\mathbin}{symbols2}{"25}

\def\;{\mathop{\typecolon}}

\newcommand{\SLE}{$\operatorname{SLE}_\kappa$}
\newcommand{\SLEk}[1]{{$\operatorname{SLE}_{#1}$}}

\numberwithin{equation}{section}
\numberwithin{thm}{section}
\usepackage{fullpage}

\author{S. C. Park
}\address{
Sung Chul Park, University of Michigan, Department of Mathematics, \texttt{scpark@umich.edu}}

\title{Conformal Invariance of the FK-Ising Model on Lorentz-Maximal S-Embeddings}

\begin{document}

\begin{abstract}
We show on non-flat but critical s-embeddings the celebrated convergence of the interface curves of the critical FK Ising model to an \SLEk{16/3} curve, using discrete complex analytic techniques first used in \cite{smirnov, sle-convergence} and subsequently extended to more lattice settings including isoradial graphs \cite{chelkak-smirnov}, circle packings \cite{lis}, and flat s-embeddings \cite{chelkak}. In our setting, the s-embedding approximates a maximal surface in the Minkowski space $\mathbb R^{2,1}$, an `exact' criticality condition identified in \cite{chelkak}, which is stronger than the percolation-theoretic `near-critical' setup studied in, e.g., \cite{mahfouf}. The proof relies on a careful discretisation of the Laplace-Beltrami operator on the s-embedding, which is crucial in identifying the limit of the martingale observable.
\end{abstract}

\maketitle
{
  \hypersetup{linkcolor=black}
}

\section{Introduction}\label{sec:intro}
The \emph{planar Ising model}, ever since its exact solution on $\mathbb Z^2$ by Onsager \cite{onsager} and celebrated rigorous computations that appeared shortly after (e.g. \cite{wu, yang, wmtb}), has been one of the most significant models in statistical mechanics. The {Fortuin-Kasteleyn} (FK) or {Random Cluster} (RC) representation of the planar Ising model is a bond percolation model which is coupled to the classical (spin) Ising model \cite{edwards-sokal}. This \emph{FK-Ising model} not only is a deeply interesting model in itself but also has allowed for significant innovations in the perspective and machinery in working with the spin model, e.g. through scale-invariant crossing estimates \cite{ising-rsw}. For general background on these models, we refer to \cite{mccoy-wu, palmer, grimmett} and references therein.

The study of these models saw a breakthrough with the introduction of \emph{Schramm-Loewner Evolution} by Schramm \cite{schramm}, the putative conformally invariant scaling limit for many two-dimensional non-self-crossing random curves at criticality. This concrete probabilistic model added to an already rich literature in theoretical physics which managed to bootstrap assumptions of local symmetry to an intricate \emph{Conformal Field Theory} \cite{bpz}. While these predictions of emergent conformal invariance have been made for numerous discrete models which hypothetically undergo a continuous phase transition, rigorous confirmations remain scarce (e.g. \cite{lsw-lerw, dgff, smirnov-percolation}).

Nonetheless, for the specific case of the critical FK- and spin-Ising models, the natural \emph{discrete holomorphicity} \cite{mercat, chelkak-smirnov-i} of Kadanoff-Ceva type \emph{fermionic correlations} \cite{kadanoff-ceva}, along with its full utilisation in the form of \emph{s-holomorphicity} \cite{smirnov-fk-martingale, smirnov}, has yielded many verifications of conformal invariance, including convergence of the interface curves to \SLEk{16/3} under scaling limit \cite{smirnov, sle-convergence}. We refer to \cite{chelkak-icm} and references therein for this considerable body of work.

Incredibly, many such results have been valid under significantly more general setups than $\mathbb Z^2$, such as isoradial graphs \cite{chelkak-smirnov} and circle patterns \cite{lis}. Recently, Chelkak \cite{chelkak-icm, chelkak} has introduced the notion of \emph{s-embeddings}, which removes exact assumptions on the specific lattice structure and instead identifies what a natural embedding is for a given abstract graph and weights. Chelkak proved in \cite{chelkak} that the interface curves of the FK-Ising model converge to \SLEk{16/3} if its s-embedding is sufficiently flat, under natural non-degeneracy conditions; see \cite[Section 5.2]{cim-zigzag}, \cite[Section 3.3]{cim} for more examples of applications.

This change of perspective, from defining the model on pre-determined lattice structure to embedding an abstract weighted planar graph in three dimensions, makes it a natural tool to study small but generic perturbations of the weights or the model in random environments. Such settings often give rise to nonzero local \emph{mass}, or finite length scale, which would disallow scale-invariant scaling limits such as SLE. Still, since such local behaviour does not meaningfully alter global phenomena, particular care is warranted about the exact notion of criticality used.

Surprisingly, it turns out that the natural ambient space for the embedded planar graph is the Minkowski space $\mathbb R^{2,1}$ and the extrinsic \emph{Lorentz geometry} of the embedded discrete surface becomes relevant in this question. For example, the surface being uniformly space-like corresponds to the model being sufficiently \emph{near-critical}, i.e. we have uniform crossing estimates \cite{mahfouf} (so e.g. the percolation function vanishes); on the other hand, (a discretised notion of) the mean curvature should measure the mass \cite[Section 2.7]{chelkak}.

That is, if the s-embeddings approximate a surface with zero mean curvature at every point, i.e. is \emph{maximal}, the model should be (approximately) scale-invariant at every point and have conformally invariant scaling limit. This is precisely the setting in which we prove convergence of the interface curves to \SLEk{16/3}. Therefore, our result concerns the general case of possibly non-flat maximal surfaces, again under natural non-degeneracy conditions. In a sense, this treats the general setting for such convergence: as long as the s-embeddings approximate a continuous surface, convergence to SLE (and local notions of conformal invariance) is not expected to hold if the surface is not maximal.

We stress that, because isothermal coordinates do not respect local s-embedding structure and therefore is not a readily exploitable tool in the discrete setting, we \emph{cannot} easily flatten an s-embedded discrete surface and trivially reduce to arguments valid for the flat setting. Contrast this with, for example, the natural (continuum) definition for \SLE\ on $\Sigma$ given in Section \ref{subsec:sle-framework}, which may even use {any} isothermal coordinates to pull back to a flat domain thanks to the inherent conformal invariance in the measure.

Indeed, our main innovation is discovering a suitable discretisation process for the (principal symbol of) Laplace-Beltrami operator in the non-flat setting, whose justification relies on local flattening (i.e. making horizontal) of the s-embedding only using global isometries of the Minkowski space and rather careful estimates of the resulting error, carried out in Section \ref{subsec:s-laplacian-calculation}.

\subsection{Critical FK-Ising model on s-embeddings}
Let $G$ be an abstract planar graph (without multiple edges, loops, and vertices of degree one), with weights $x_e\in[0,1]$ assigned on edges $e\in \lozenge(G)$. The sets of vertices and edges of this graph are respectively written $G^\bullet$ and $\lozenge(G)$. By fixing an `unbounded face' (to be excluded), we consider the set of faces $G^\circ$, or the vertex set of the dual graph, with the same set of edges $\lozenge(G)$. We visualize each edge, or a \emph{quad}, $e\in \lozenge(G)$ as the abstract quadrilateral $(v_0^\bullet v_0^\circ v_1^\bullet v_1^\circ)$ bounded by the two primal and dual vertices (respectively $v_0^\bullet,  v_1^\bullet$ and $v_0^\circ,  v_1^\circ$) that it is incident to; specifying $e$ as either primal or dual simply dictates which pair are being connected by $e$. Let us also define $\Upsilon(G)$ as the set of medial edges, i.e. pairs $(v^\bullet v^\circ)$ where $v^\bullet$ and $v^\circ$ are incident, often identified with the midpoint of the two vertices; $(v^\bullet v^\circ)$ and $({w^\bullet} {w^\circ})$ are adjacent if exactly one of $v^\bullet = {w^\bullet}$ or $v^\circ = {w^\circ}$ holds. 

The \emph{FK-Ising} model on $(G, x)$ is the probability measure $\mathbb P$, on subsets $\mathsf{P}\subset E$, defined by
\begin{equation}\label{eq:fk-def}
    \mathbb P_{\text{FK}}[\mathsf P]\propto 2^{\#\operatorname{clusters}(\mathsf P)}\prod_{e\in \mathsf P} \frac{1-x_e}{x_e},
\end{equation}
where $\#\operatorname{clusters}(\mathsf P)$ is the number of vertex clusters connected by $\mathsf P$ as primal edges.

  There is an associated \emph{dual} configuration $\mathsf D := \lozenge(G) \setminus \mathsf P$ to each $\mathsf P$ (see Figure \ref{fig:lattice}). Given $(\mathsf P, \mathsf D)$, the \emph{interface curve} between a primal vertex cluster $\mathtt p$ connected by $\mathsf P$ and a dual vertex cluster $\mathtt d$ connected by $\mathsf D$ is the nearest-neighbor sequence of medial edges $(v^\bullet v^\circ)\in \Upsilon(G)$ such that we have $v^\bullet \in \mathtt p$ and $v^\circ \in \mathtt d$ (once an embedding is fixed, we visualize it as a smooth curve interpolating between the midpoint of each pair ($v^\bullet v^\circ$) to the next one, perpendicularly to the medial edges).

An \emph{s-embedding} of the pair $(G, x)$ is defined as any embedding (i.e. map, but we often identify an abstract graph element with its embedding) $\mathcal S:G^\bullet\cup G^\circ \to \mathbb C$ satisfying the following conditions: if $e=(v_0^\bullet v_0^\circ v_1^\bullet v_1^\circ)\in \lozenge(G)$, then
\begin{itemize}
    \item $\mathcal S(v_0^\bullet), \mathcal S(v_0^\circ), \mathcal S(v_1^\bullet), \mathcal S(v_1^\circ)\in \mathbb C$ serve as the corners of a quadrilateral in $\mathbb C$ that is \emph{tangential} (i.e. has an circumscribed circle); the centre of this circle is henceforth identified with $\mathcal S(e)$, extending $\mathcal S$ to $\lozenge(G)$.
    \item The weight $x_e$ is reflected in the quad by the formula
    \begin{equation}
        \tan\theta_e = \frac{|\mathcal{S}(v_0^\bullet)-\mathcal{S}(e)|\cdot |\mathcal{S}(v_1^\bullet)-\mathcal{S}(e)|}{|\mathcal{S}(v_0^\circ)-\mathcal{S}(e)|\cdot |\mathcal{S}(v_1^\circ)-\mathcal{S}(e)|},
    \end{equation}
    where $\theta_e\in [0,\frac\pi2]$ is defined by $\theta_e = 2\arctan x_e$.
\end{itemize}

$\mathcal S$ is called \emph{proper} if no two tangential quads as identified above overlap with each other. The tangential quad condition also allows us to lift $\mathcal S \in \mathbb C \cong \mathbb R^2$ to $(\mathcal S, \mathcal Q)\in \mathbb R^{2,1}$, by the following definition of $\mathcal Q:G^\bullet\cup G^\circ \to \mathbb R$. Given any $(v^\bullet v^\circ)\in \Upsilon(G)$, we define (note the asymmetry between the primal and dual graphs)
$$
\mathcal Q(v^\bullet) - \mathcal Q(v^\circ) = |S(v^\bullet) - S(v^\circ)| \geq 0.
$$
An equivalent condition for a quadrilateral to be tangential is that the two sums of opposite sides should coincide, exactly showing the well-definedness of $\mathcal Q$ around a single quad, and therefore everywhere. Note that $\mathcal S$ uniquely determines $\mathcal Q$ up to an additive constant. One hint that the ambient space should be taken with the Minkowski signature is the following \emph{Lorentz invariance} property: if we map an s-embedding $(\mathcal S, \mathcal Q)$ under an isometry of $\mathbb R^{2,1}$, the resulting embedded graph is still an s-embedding of the same weighted graph (see \cite[Remark 1.2]{chelkak}).

We consider space-like surfaces $\vartheta:\Omega \subset \mathbb R^2 \to \mathbb R$ whose second derivatives are bounded up to the boundary. This holds, in particular, if $\Omega$ is smooth and $\vartheta$ is maximal, i.e. has zero mean curvature as embedded in $\mathbb R^{2,1}$ (see also \eqref{eq:mean-curvature-vector}), with sufficiently regular boundary data \cite{minimal-surfaces}, but note that it is then also {inherited} by any possibly rough subdomain of $\Omega$, e.g. slit by a path. We state our regularity and criticality assumptions in terms of $\mathcal Q$. Assume we have a family of s-embeddings $(\mathcal{S}^\delta, \mathcal Q^\delta)$ for small $\delta>0$, and  define the following two conditions, both of which depend on a ($\delta$-uniform) constant $C>0$:
\begin{enumerate}
\item[$\operatorname{UNIF}(\delta)$:] we have $|S(v^\bullet) - S(v^\circ)| \in \left[C^{-1}{\delta},C{\delta}\right]$ for all medial edges $(v^\bullet v^\circ)\in \Upsilon(G)$, and all angles of the tangential quads are bounded below by $C^{-1}$.
\item[$\operatorname{Approx}(\vartheta,\delta)$:] we have $|\mathcal Q(v) - \vartheta\circ\mathcal S(v)| \leq C\delta$ for all $v\in (G^\bullet\cup G^\circ)$ such that $\mathcal S(v)\in \Omega$.
\end{enumerate}
$\operatorname{Approx}(\vartheta,\delta)$ is the conceptual condition asserting sufficient criticality for SLE convergence, and exactly reduces to the condition $\operatorname{FLAT}(\delta)$ in \cite[Section 1.2]{chelkak} when $\vartheta \equiv 0$. $\operatorname{UNIF}(\delta)$ is the regularity assumption, implying in particular a weaker condition $\operatorname{Lip}(\kappa,\delta)$ (see \cite[(1.12)]{chelkak} and discussion below) which discretises (and implies, through scaling limit) the non-degeneracy condition, being \emph{uniformly space-like}, of the surface $\vartheta$, i.e. $|\nabla \vartheta| \leq \kappa < 1$.

We now state our main result. To define the FK-Ising model with \emph{(Dobrushin) wired/free boundary conditions}, we assume we have for each $\delta$ nearest-neighbour, counter-clockwise sequences of primal vertices $( b^\delta a^\delta)^\bullet = \{v_1^\bullet , v_2^\bullet , \cdots , v_n^\bullet \}$ and dual vertices $(a^\delta b^\delta)^\circ =\{v_1^\circ , v_2^\circ , \cdots , v_{n'}^\circ\}$, with medial edges $a^\delta = (v_n^\bullet v_{1}^\circ)$ and $b^\delta = (v_{1}^\bullet v_{n'}^\circ)$ such that, on the s-embedding $\mathcal S^\delta$, the medial edges $a^\delta,b^\delta$ as well as the primal edges and dual edges traversed in $( b^\delta a^\delta)^\bullet,(a^\delta b^\delta)^\circ $ bound a simply connected polygonal domain $\Omega^\delta \subset \Omega$. Then we define the FK-Ising model with wired/free boundary conditions on $(\Omega^\delta,a^\delta,b^\delta)$ by considering the definition \eqref{eq:fk-def} on only those $\mathsf P$ that contain the boundary primal edges and not containing the boundary dual edges. Note that there are unique boundary primal and dual clusters $\mathsf{p}_b, \mathsf{d}_b$ defined this way, and a unique interface curve $\gamma$ connecting $a^\delta$ to $b^\delta$.

 It is standard that martingale convergence as in Theorem \ref{thm:martingale-convergence} and a Russo-Seymour-Welsh type crossing estimates, given in our case by \cite{mahfouf}, imply convergence to SLE, our main result; see e.g. \cite{sle-convergence, kemppainen-smirnov} and \cite[Remark 1.4]{chelkak} for more details. We also review the definition of \SLE\ on surfaces in the next section.
\begin{thm}\label{thm:interface-convergence}
    Suppose a family of s-embeddings $(\mathcal{S}^\delta, \mathcal Q^\delta)$ satisfies $\operatorname{UNIF}(\delta)$ and $\operatorname{Approx}(\vartheta,\delta)$ for some maximal $\vartheta:\Omega \subset \mathbb C \to \mathbb R$. We also assume the marked domain $(\Omega^\delta,a^\delta,b^\delta)$ converges to $(\Omega,a,b)$ as $\delta\downarrow 0$ in the Carath\'eodory sense (see \cite[Section 3.2]{chelkak-smirnov-i}). Then the interface curve from $a^\delta$ to $b^\delta$ converges to \SLEk{16/3} (see \cite{sle-convergence} or \cite{kemppainen-smirnov} for the precise mode of convergence).
\end{thm}
\begin{figure}[h]
    \centering
    \includegraphics[width=\textwidth]{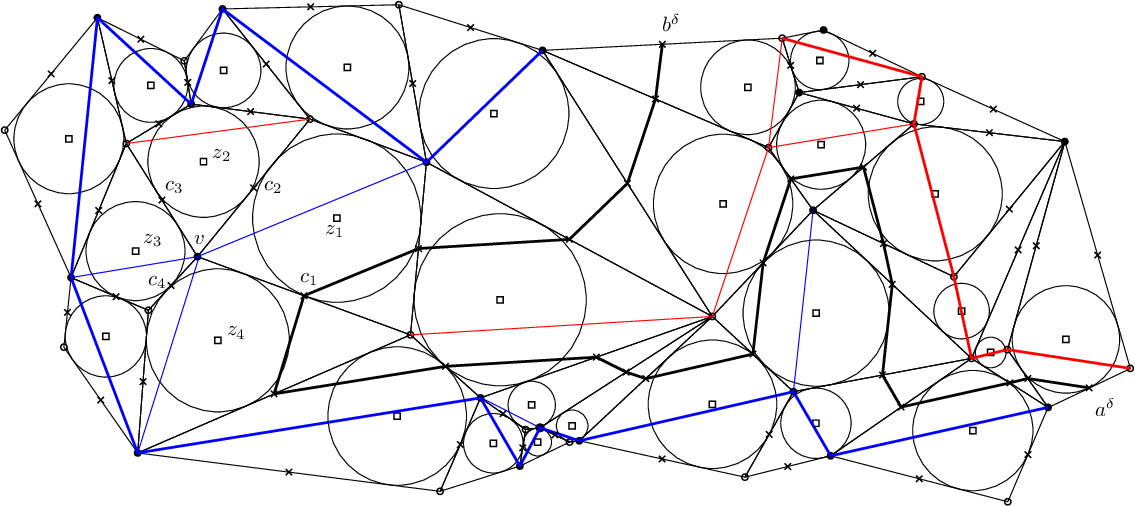}
    \caption{An example of s-embedded graph. Elements of $G^\bullet$, $G^\circ$, $\lozenge(G)$, and $\Upsilon(G)$ are respectively marked with $\bullet$, $\circ$, $\square$, and $\times$. Thick blue and red edges are respectively the primal and dual edges fixed by the boundary condition; thinner blue and red edges are the primal and dual edges specific to this particular configuration, the sampled blue edges $\mathsf{P}$ completely determining the red edges $\mathsf{D}$. The thick black path is the interface $\gamma$.}
    \label{fig:lattice}
\end{figure}
\subsection{SLE convergence framework}\label{subsec:sle-framework}
Given a real parameter $\kappa>0$, recall that \SLE\ on the upper half-plane $\H$ from $0$ to $\infty$ is a probability measure on random non-crossing paths from $0$ to $\infty$. It is characterised by quantities $g_t(z)$, for each $z\in \mathbb H$, defined by the initial condition $g_0(z)=z$ and the \emph{Loewner differential equation}
$$\partial_tg_t(z) = \frac{2}{g_t(z)-\sqrt{\kappa}B_t},$$
where $B_t$ is the standard Brownian motion. As a function of $z$, almost surely, each $g_t$ is a conformal map (random depending on the trace $B_{[0,t]}$) from $\mathbb H\setminus K_t$ for some curve $K_t \subset \mathbb H$ to $\mathbb H$, and the tip of the \SLE\ trace is defined as $\gamma_t := \lim_{z\to \sqrt{\kappa}B_t} g_t^{-1}(z)$.

\SLE\ on other simply connected domains $\Omega \subsetneq \C$ with marked boundary points (prime ends) $a, b\in \partial \Omega$ is defined as the pull-back of the \SLE\ measure by any Riemann map from $\Omega \to \H$ taking $a$ to $0$ and $b$ to $\infty$, which is well-defined thanks to the inherent conformal invariance in the law. In fact, we may define \SLE\ on any simply connected Riemannian surface with boundary, with marked boundary points, as the pull-back of the \SLE\ measure on its conformal parametrization (see Section \ref{subsec:continuum-geometry}).

The most common strategy to prove convergence of random discrete interfaces to \SLE\ has been through the so-called \emph{martingale observables} $F^\delta(\cdot;\Omega^\delta ,a^\delta,b^\delta)$. These are functions typically defined on parts of $\Omega$ or $\partial \Omega$ and parametrised by marked domains, taken as the domains $(\Omega_n^\delta,a_n^\delta,b^\delta)$ slit by the first $n$ steps of the discrete interface, such that for fixed $z$ we have that $F^\delta(z;\Omega_n^\delta,a_n^\delta,b^\delta)$ is a martingale. 

Concretely, consider the FK-Ising model with wired/free boundary conditions on $(\Omega^\delta,a^\delta,b^\delta)$. Given any medial edge $c = (c^\bullet c^\circ)$ on $\Omega^\delta$, we consider the following quantity, which goes back to Smirnov \cite{smirnov-fk-martingale, smirnov}: recall $b^\delta=(\mathcal{S}(v_1^\bullet)  \mathcal{S}(v_{n'}^\circ))$, then consider the following, which is well-defined up to a global sign,
\begin{equation}\label{eq:smirnov}
F^\delta(c;\Omega^\delta,a^\delta,b^\delta)= \mathbb E_{\text{FK}}\left[\frac{\mathbf{1}_{c\in \gamma}e^{-\frac{i\pi}{2}\#W(\gamma,c,b^\delta)}}{\sqrt{\operatorname{n}(b^\delta)}\sqrt{|\mathcal{S}(c^\bullet) - \mathcal{S}(c^\circ)|}} \right],\text{ where }\operatorname{n}(b^\delta):= \frac{\mathcal{S}(v_1^\bullet) - \mathcal{S}(v_{n'}^\circ)}{|\mathcal{S}(v_1^\bullet) - \mathcal{S}(v_{n'}^\circ)|},
\end{equation}
$\gamma$ is the interface curve, and $W(\gamma)$ is the total turning of the tangent vector of $\gamma$ from the midpoint of $c$ to that of $b^\delta$ (recall the medial edges are traversed perpendicularly). Note that the phase factors ${{\operatorname{n}(b^\delta)}^{-\frac12}}{e^{-\frac{i\pi}{2}\#W(\gamma,c,b^\delta)}}$ implies that $F^\delta(c;\Omega^\delta,a^\delta,b^\delta) \in \operatorname{n}(c)^{-\frac12}\mathbb R$. We elaborate more on this discrete function in Section \ref{subsec:discrete-geometry}.

Convergence of this quantity to an analogous martingale in continuum is sufficient to uniquely identify the limiting distribution as \SLEk{16/3}. Indeed, the continuum limit $f(\cdot;\Omega,a,b)$ is the natural pull-back of the observables employed in \cite{smirnov,sle-convergence,chelkak} for the s-embedding setting; in terms of isothermal coordinates
$$\mathbb R\times (0,1) \xrightarrow{\zeta \mapsto z(\zeta)} \Omega,\,z(\infty)=a,\,z(-\infty)=b$$
that conformally parametrize the maximal surface $\vartheta$, we have
$$
f(z(\zeta);\Omega,a,b) = \frac{ \overline{(\partial_{\zeta} z(\zeta))^{1/2}} - i (\partial_{\zeta} \bar z(\zeta))^{1/2}}{|\partial_{\zeta} z(\zeta)| - |\partial_{\zeta} \bar z(\zeta)|}.
$$
See Section \ref{subsec:continuum-geometry} for more context on the continuum martingale.

Write $\operatorname{Proj}_{e^{i\theta}\mathbb R}A = \frac{A+e^{2i\theta}\bar A}{2}$ for the orthogonal projection of $A\in \mathbb C$ onto the line $e^{i\theta}\mathbb R$.
\begin{thm}\label{thm:martingale-convergence}
    Suppose a family of s-embeddings $(\mathcal{S}^\delta, \mathcal Q^\delta)$ satisfies $\operatorname{UNIF}(\delta)$ and $\operatorname{Approx}(\vartheta,\delta)$ for some maximal $\vartheta:\Omega \subset \mathbb C \to \mathbb R$. We also assume the marked domain $(\Omega^\delta,a^\delta,b^\delta)$ converges to $(\Omega,a,b)$ as $\delta\downarrow 0$ in the Carath\'eodory sense (see \cite[Section 3.2]{chelkak}). Then the discrete martingale observable $F^\delta(\cdot;\Omega^\delta,a^\delta,b^\delta)$ converges to $f(\cdot;\Omega,a,b)$, in the sense that,
        $$
         F^\delta(c;\Omega^\delta,a^\delta,b^\delta) = \operatorname{Proj}_{n(c)^{-\frac12}\mathbb R}f(c;\Omega,a,b) + o(1),
        $$
    uniformly for medial vertices $c$ on compact subsets of $\Omega$. 
\end{thm}
\begin{proof}
    See Section \ref{subsec:discrete-geometry}. There we note that there must be some subsequential limit $f_d$ for the functions $F^\delta$, so it remains to uniquely identify the limit as $f$. The main ingredient in doing so is Proposition \ref{prop:laplace-beltrami}, proved in Section \ref{subsec:s-laplacian-calculation}, showing that its \emph{(square-)integral} $h_d$, defined as in \eqref{eq:continuum-integral}, takes the boundary values \eqref{eq:continuum-h-definition}.
\end{proof}

\begin{remark}
    A simple but rich class of non-flat maximal s-embeddings is given by triangulations of $\Omega$, whose vertices are taken to be $G^\bullet$. After letting $\mathcal{Q}^\delta = \vartheta$ at these primal vertices, $(\mathcal S^\delta, \mathcal Q^\delta)$ of the dual vertex within any triangular face is uniquely determined (e.g. under an isometry sending the triangle to a horizontal one, the dual vertex is sent to the circumcentre). It is elementary to see then the resulting s-embedding is $\operatorname{UNIF}(\delta)$, constant depending on $\kappa$, as long as the initial triangulation had uniformly bounded angles and lengths.
\end{remark}

\begin{remark}
    The arguments going into Theorem \ref{thm:martingale-convergence} does not use maximality of $\vartheta$ except to define the continuum martingale uniquely, using a Dirichlet boundary problem \eqref{eq:continuum-h-definition} and the Laplace-Beltrami operator $\Delta^\Sigma$. That is, we may trivially extend the Theorem to massive settings where we have uniqueness of Dirichlet boundary problem for, say, the following equation on the conformal parametrisation (see \cite[Section 2.7]{chelkak} and \cite[Section 5.1]{park})
    $$
    \Delta_\zeta (h\circ z) = 2{(H\circ z)\cdot l}\cdot|\nabla (h\circ z)|,
    $$
    where $l$ and $H$ are respectively the metric element and the mean curvature (see Section \ref{subsec:continuum-geometry}). 
\end{remark}

\subsection{Outline of the paper}
The rest of this paper is organised as follows. We first recall the basic language of embedded space-like surfaces in $\mathbb R^{2,1}$ in Section \ref{subsec:continuum-geometry}. Section \ref{subsec:discrete-geometry} recalls known fundamental properties and the necessary analytic framework for the discrete martingale observable, mostly established in \cite{chelkak}, in particular explaining that Theorem \ref{thm:martingale-convergence} is proved given Proposition \ref{prop:laplace-beltrami}. Then we go on to prove Proposition \ref{prop:laplace-beltrami} in Section \ref{subsec:s-laplacian-calculation}.

\medskip

{\bf Acknowledgments:} The author is indebted to Dmitry Chelkak for helpful conversations and advices, which was crucial especially for this work. He also thanks R\'emy Mahfouf for many foundational conversations on this topic. This work was partially accomplished while the author was employed as Research Fellow at Korea Institute for Advanced Study (KIAS), supported by a KIAS Individual Grant (MG077202).

\section{Convergence of the Martingale Observable}

\subsection{Lorentz geometry and continuum observables}\label{subsec:continuum-geometry}

We now briefly recall the elements of Lorentz geometry needed to uniquely define the continuum functions that the martingale observables converge to. See \cite{geometry-reference} for general reference on extrinsic geometry in pseudo-Riemannian manifolds, as well as \cite[Section 2.7]{chelkak} for analysis specific to this case.

Consider a space-like surface $\vartheta:\Omega\subset \mathbb R^2 \to \mathbb R$ embedded as $\Sigma \subset \mathbb R^{2,1}$. If $\vartheta$ is sufficiently regular (as it is if $\Sigma$ is maximal), the \emph{Laplace-Beltrami operator} $\Delta^\Sigma$ is given in terms of the coordinates $z= x_1 + ix_2$ and metric $g$ as the following
\begin{equation}\label{eq:laplace-beltrami-def}
    \Delta^\Sigma h = \frac{1}{\sqrt{|g|}}\partial_{x_i}\left(\sqrt{|g|}g^{ij}\partial_{x_j} h\right).
\end{equation}
Observing this formula, let us define the \emph{tangent plane Laplacian} $\Delta^{\Sigma(a)}$ for $a\in \Omega$ as the Laplace-Beltrami operator on the tangent plane of $\Sigma$ at $(a, \vartheta(a))$, having constant coefficients. Note that there is a first-order operator $D^{\Sigma(a)}$ with constant coefficients such that we have
\begin{equation}\label{eq:tangent-plane-def}
    \Delta^\Sigma h(a) = \Delta^{\Sigma(a)}h(a) + D^{\Sigma(a)}h(a).
\end{equation}

There is a conformal parametrisation (or isothermal coordinates) $\zeta = \zeta_1 + i\zeta_2\in U\subset \mathbb C \mapsto (z(\zeta),\vartheta \circ z(\zeta))) \in \Sigma$, in which the metric on $\Sigma$ is described by a single isotropic element $l(\zeta) = (|\partial_{\zeta} z|^2+|\partial_{\zeta} \bar z|^2-2|\partial_{\zeta} (\vartheta\circ z)|^2)^{1/2} = |\partial_{\zeta} z| - |\partial_{\zeta} \bar z| >0$, since $(\partial_{\zeta} (\vartheta\circ z))^2 = (\partial_{\zeta} z)\cdot (\partial_{\zeta} \bar z)$ under the conformal parametrization assumption (see \cite[(2.27)]{chelkak}).

Then $\Delta^\Sigma$ on $\Sigma$ may be related to the flat laplacian $\Delta_{\zeta}$ on $U$ by
\begin{align}
    \Delta_{\Sigma} = l^{-2}\Delta_{\zeta}.
\end{align}
In terms of this operator, the mean curvature vector (that is, the scalar mean curvature times the unit normal vector) in $\mathbb R^{2,1}$ is written
\begin{align}\label{eq:mean-curvature-vector}
    \vec H = (\Delta^\Sigma z,\Delta^\Sigma (\vartheta\circ z)).
\end{align}
In particular, the components of a conformal parametrization are harmonic if and only if the surface is maximal.

Given a marked domain $(\Omega, a, b)$ and a maximal surface $\vartheta:\Omega \to \mathbb R$, there is a unique harmonic function, the \emph{(square-)integrated continuum martingale} $h = h(\cdot;\Omega,a,b):\Omega \mapsto [0,1]$ such that
\begin{equation}\label{eq:continuum-h-definition}
\begin{cases}
    \Delta^\Sigma\hspace{-1em}&h = 0 \text{ in }\Omega,\\
    &h = 0\text{ on }(ab),\\
    &h  = 1\text{ on }(ba),
\end{cases}
\end{equation}
which, e.g., may be easily seen using the conformal parametrization.

Then, the \emph{(conformal domain) continuum martingale} $\phi = \phi(\cdot;U,z^{-1}(a),z^{-1}(b)):U \to \mathbb C$, defined up to a global sign, is the holomorphic function $\phi(\zeta) := \sqrt{\partial_{\zeta}h(z(\zeta),\vartheta\circ z(\zeta))}$. Note that it is straightforward from the uniqueness of $h$ to verify that $\phi$ coincides with $\sqrt{\Phi'}$, where $\Phi$ is the conformal map mapping $U$ to the strip $\mathbb R\times(0,1)\subset \mathbb C$, with the marked points $z^{-1}(a)$ and $z^{-1}(b)$ respectively mapped to $\infty$ and $-\infty$.

If $U_t$ is the domain $(U_t,z^{-1}(a_t),z^{-1}(b))$ slit by \SLEk{16/3} from $z^{-1}(a)$ to $z^{-1}(a)$, it is standard (see, e.g. \cite{smirnov-fk-martingale}) that $t\mapsto \phi(\zeta_0 ;U_t,z^{-1}(a_t),z^{-1}(b))$ is a martingale for any fixed $\zeta_0\in U$ while $\zeta_0$ remains in $U_t$; of course, given our definition of \SLE\ using pullbacks, we may equivalently describe $U_t$ as the conformal parametrization of the domain $(\Omega_t, a_t, b)$ slit by \SLEk{16/3} on $(\Omega, a, b)$ (here we use the conformal parametrization $\zeta \mapsto z(\zeta)$ for $t=0$ for all times $t$).

On the physical (s-embedded) domain $\Omega$ though, we instead consider the \emph{(physical domain) continuum martingale} $f=f(\cdot;\Omega,a,b):\Omega \to \mathbb C$ defined by the relation
\begin{equation}\label{eq:phi-definition}
\phi = (f\circ z)\cdot(\partial_\zeta z)^{1/2} + i (\bar f\circ z)\cdot(\partial_\zeta \bar z)^{1/2}\text{, or }f\circ z = \frac{\phi \overline{(\partial_{\zeta} z)^{1/2}} - i\bar \phi (\partial_{\zeta} \bar z)^{1/2}}{|\partial_{\zeta} z| - |\partial_{\zeta} \bar z|}.
\end{equation}
By linearity, it is clear that $f(z_0 ;\Omega_t,a_t,b)$ is also a martingale for any fixed $z_0$ and $(\Omega_t, a_t, b)$ slit by \SLEk{16/3} as above. While $f$ may be obtained from $h$ by going to the conformal parametrization, we note the following direct relation between them:
\begin{align} \nonumber
    h\circ z &= \Im \int \frac{\phi^2}{2} d\zeta\\ \nonumber
    &= \Im \int \frac{(f\circ z)^2(\partial_\zeta z) - (\bar f\circ z)^2(\partial_\zeta \bar z) + 2i|f\circ z|^2(\partial_\zeta z)^{1/2}(\partial_\zeta \bar z)^{1/2}}{2} d\zeta\\ \label{eq:continuum-integral}
    &=\Im \int \partial_\zeta\Im\int f^2 dz +i|f\circ z|^2(\partial_\zeta \vartheta) d\zeta=\Im\int \frac{1}{2} \left(f^2dz +i|f|^2d \vartheta\right).
\end{align}

\subsection{Analysis of the discrete observable}\label{subsec:discrete-geometry}
While the definition \eqref{eq:smirnov} is fully self-contained, we find it helpful to interpret the observable as a correlation in the physical model. Note that $\mathbf{1}_{z \in \gamma} = \mathbf{1}_{c^\bullet \in \mathsf{p}_b}\mathbf{1}_{c^\circ \in \mathsf{d}_b}$, i.e. the indicator of simultaneous primal-dual connection of $c$ to respective boundary clusters. Indeed, it is known that $$F^\delta(c;\Omega^\delta,a^\delta,b^\delta) = \operatorname{n}(c)^{\frac12}\mathbb E[ \sigma_{\mathsf{p}_b}\mu_{\mathsf{d}_b}\sigma_{c^\bullet}\mu_{c^{\circ}}],$$ where $X_{c}:=\mathbb E[ \sigma_{\mathsf{p}_b}\mu_{\mathsf{d}_b}\sigma_{c^\bullet}\mu_{c^{\circ}}] \in \mathbb R$ is an (embedding-agnostic) correlation of the Kadanoff-Ceva fermions $\sigma_{c^\bullet}\mu_{c^{\circ}}$ \cite{kadanoff-ceva} for the corresponding spin-Ising model. The sign of $X_c$ is a priori ambiguous (e.g. note that we need to fix a sign of $\operatorname{n}(c)^{\frac12}$ for each $c$), but there is a way to define it as a single-valued function on a suitable \emph{double cover} of the medial edge set (see \cite[Sections 1.2, 2.1]{chelkak} and, e.g. \cite[Figure 2.1]{chi2}).

In particular, $X_c^2$ has no sign ambiguity. If $e=(v_0^\bullet v_0^\circ v_1^\bullet v_1^\circ)$ is an embedded quad in $\Omega^\delta$, then we have
\begin{equation}\label{eq:propagation-squared}
X_{(v_0^\bullet v_0^\circ)}^2 + X_{(v_1^\bullet v_1^\circ)}^2 = X_{(v_1^\bullet v_0^\circ)}^2 + X_{(v_0^\bullet v_1^\circ)}^2,
\end{equation}
by the so-called propagation equation \cite[(1.5)]{chelkak}. Then we  define the \emph{(square-)integrated martingale observable} $H^\delta :G^\bullet\cup G^\circ \to \mathbb R$, up to an additive constant, similarly to $\mathcal Q$: given any medial edge $(v^\bullet v^\circ)$, we define
\begin{equation}\label{eq:discrete-h-def}
H^\delta(v^\bullet) - H^\delta(v^\circ) = X_{(v^\bullet v^\circ)}^2 \geq 0,
\end{equation}
where $H^\delta$ is well-defined again thanks to \eqref{eq:propagation-squared}. In addition, the wired/free boundary conditions for the FK-Ising model translates to discrete boundary conditions exactly mirroring (after choosing the additive constant appropriately) the boundary conditions \eqref{eq:continuum-h-definition}: we have \cite[(2.18)]{chelkak}
\begin{equation}\label{eq:discrete-h-boundary}
H^\delta(v^\circ) = 0\text{ for }v\in (a^\delta b^\delta)^\circ\text{\,\,\,\, and\,\,\,\, }H^\delta(v^\bullet) = 1\text{ for }v\in ( b^\delta a^\delta)^\bullet.
\end{equation}
In addition, thanks to the maximum principle for $H^\delta$, we have a global bound $|H^\delta| \leq 1$ \cite[Proposition 2.11, Corollary 2.12]{chelkak}.

The propagation equation is equivalent to the notion of \emph{s-holomorphicity}, first introduced in \cite{smirnov}; in the notation above, it says that there is a unique value $F(e)$ for the quad itself, such that $F(c)=\operatorname{Proj}_{n(c)^{-\frac12}\mathbb R}F(e)$ for the 4 incident medial edges $c = (v_0^\bullet v_0^\circ),(v_1^\bullet v_1^\circ),(v_1^\bullet v_0^\circ),(v_0^\bullet v_1^\circ)$ \cite[Proposition 2.5, Corollary 2.6]{chelkak}  (note that Theorem \ref{thm:martingale-convergence} then becomes equivalent to the statement that $F^\delta(\cdot;\Omega^\delta,a^\delta,b^\delta) \to f(\cdot;\Omega,a,b)$ uniformly on compacts).

In terms of the values $F(e)$, we may verify from \eqref{eq:discrete-h-def} the following discrete integral formulation \cite[(2.17)]{chelkak}
\begin{equation}\label{eq:discrete-integral}
 d^\delta H^\delta= \Im \frac{1}{2} \left(F^\delta(e)^2 d^\delta S +i|F^\delta(e)|^2d^\delta \mathcal Q\right),
\end{equation}
where $d^\delta$ refers to a difference across the quad (say, the value at $v_0^\bullet$ minus $v_0^\circ$, etc.). Note this is clearly the discrete counterpart of \eqref{eq:continuum-integral}.

The primary ingredient in showing pre-compactness under scaling limits is as follows. While the following is a general statement for s-holomorphic functions, we state it for $H^\delta$ using the global bound $|H^\delta| \leq 1$; a simple re-scaling recovers the general bound. We now consider $F^\delta$ as being defined on $\Omega^\delta$ by, say, piecewise affine interpolation.
\begin{proposition}[{\cite[(2.23), Theorem 2.17]{chelkak}}]\label{prop:regularity}
    Suppose a family of s-embeddings $(\mathcal{S}^\delta, \mathcal Q^\delta)$ satisfies $\operatorname{UNIF}(\delta)$. Then there is are constants $C>0,\beta>0$ (depending on the $\operatorname{UNIF}(\delta)$ constant) such that for any $z_1, z_2\in B_{r}(z_0)\subset B_{2r}(z_0) \subset \Omega^\delta$, we have
    \begin{equation}\label{eq:regularity}
        |F^\delta(z_1)| \leq \frac{C}{\sqrt{r}}\text{\,\,\,\, and \,\,\,\,}|F^\delta(z_1)-F^\delta(z_2)| \leq {C}\frac{|z_1-z_2|^\beta}{r^{\frac12+\beta}}.
    \end{equation}
\end{proposition}
By Arzel\`a-Ascoli, this means there is a subsequential limit $f_0$ on any $\Omega_0 \Subset \Omega$ to which $F^\delta$ converges uniformly; by diagonalizing, we may obtain a limit $f_d$ towards which $F^\delta$ converges uniformly on compact subsets along a subsequence. If we can demonstrate $f_d = f$, we will have Theorem \ref{thm:martingale-convergence}.

Thanks to uniform convergence, it is clear that $H^\delta$ (again, as a function on $\Omega^\delta$) also converges uniformly on compacts to some $h_d = \Im \int f_d^2dz+i|f_d|^2 d\vartheta \in [0,1]$. In fact, \cite[Proposition 2.21]{chelkak} (and the discussion below) observes that $\phi_d:U\to\mathbb C$ defined under \eqref{eq:phi-definition} is holomorphic, since $\vartheta$ is maximal, therefore $h_d\circ z = \Im \int \phi_d^2 d\zeta$, as well as $h_d$ itself (by \eqref{eq:continuum-h-definition}), is harmonic. Then the following Proposition, proved next section, would imply $h_d = h$, therefore $f_d = f$:
\begin{proposition}\label{prop:laplace-beltrami}
    Near the boundary arcs $(ab), (ba)$, the subsequential limit $h_d$ continuously takes the boundary values \eqref{eq:continuum-h-definition}.
\end{proposition}

\subsection{Proof of Proposition \ref{prop:laplace-beltrami}}\label{subsec:s-laplacian-calculation}
We adapt the strategy in \cite[Section 4]{chelkak}; we will not elaborate on details already explained there to the same extent. We write $\epsilon, \eta, \gamma,\alpha>0$ for exponents to be specified later.

First of all, we mollify the discrete function $H^\delta(u)$ on the (variable) radius $\rho_u:=\delta^\epsilon \operatorname{crad}(u,\Omega^\delta)$ with $\epsilon \ll 1$ for $u \in \Omega_{\operatorname{int}(\eta)}^\delta$, where $\Omega_{\operatorname{int}(\eta)}^\delta$ is the $\delta^{1-\eta}$-interior of $\Omega^\delta$ and $\operatorname{crad}(u,\Omega^\delta) \asymp d_u:=\operatorname{dist}(u,\partial\Omega^\delta)$ is the conformal radius. Taking a standard unit radius mollifier $\varphi_0$ and re-scaling $\varphi^u(w):= \rho_u^{-2}\cdot\varphi_0(\rho_u^{-1}\cdot(w-u))$, we have the following definition and subsequent bound \cite[(4.2), (4.3)]{chelkak},
\begin{equation}\label{eq:mollifier}
\tilde{H}^\delta(u):= \int_{\Omega^\delta}\varphi(w,u)H^\delta(w)d^2w,\text{\,\,\,\,and\,\,\,\,}|H^\delta-\tilde{H}^\delta| = O(\delta^\epsilon) \text{ uniformly on }\Omega_{\operatorname{int}(\eta)}^\delta.
\end{equation}
Then, we seek a uniform bound on $u\in \Omega_{\operatorname{int}(\eta)}^\delta$ of the type {\cite[(4.5)]{chelkak}}
\begin{equation}\label{eq:laplacian-bound}
\Delta^\Sigma \tilde{H}^\delta(u) = O(\delta^\alpha d_u^{-2+\alpha})
\end{equation}
which precisely controls the continuity of $\tilde{H}^\delta(u)$ up to the boundary (e.g. see \cite[Lemma A.2]{chelkak}). Then, the remaining boundary continuity argument for $H^\delta$ (and therefore its subsequential limit $h_d$), as in \cite[Section 4.1]{chelkak}, is to use the boundary regularity of $\tilde H^\delta$ given by \eqref{eq:laplacian-bound}, the known boundary values of $H^\delta$, and the closeness of these two functions \eqref{eq:mollifier}. This is not sufficient since we are still in the strictly smaller domain $\Omega_{\operatorname{int}(\eta)}^\delta \subsetneq \Omega^\delta$ (i.e. the values of $H^\delta$ on $\partial\Omega_{\operatorname{int}(\eta)}^\delta$ may a priori not be close to its values on $\partial\Omega^\delta$), but see \cite[Section 4.3]{chelkak} for argument bootstrapping RSW-type estimates (shown in \cite{mahfouf} for non-flat embeddings) to show boundary regularity of $H^\delta$ on thin strips near $\partial \Omega^\delta$ (which would include $\partial\Omega_{\operatorname{int}(\eta)}^\delta$ if we choose small enough $\eta>0$), completing the proof of Proposition \ref{prop:laplace-beltrami}.

What remains is to show a bound of type \eqref{eq:laplacian-bound}. We need to work with functions defined on $G^\bullet \cup G^\circ$ and/or $\lozenge (G)$: we start by defining some natural discrete differential operators here.
\begin{definition}
    Given a function $I$ locally defined on $G^\bullet \cup G^\circ$, we define $\bar \partial_{{\mathcal S}}I$, a function locally defined on $\lozenge G$, by the following formula: if $z = (v_0^\bullet v_0^\circ v_1^\bullet v_1^\circ)$,
    $$
    \bar \partial_{{\mathcal S}}I(z) := \frac{\mu_z}{4}\left[\frac{I(v_0^\bullet)}{\mathcal S(v_0^\bullet)-\mathcal S(z)}+ \frac{I(v_1^\bullet)}{\mathcal S(v_1^\bullet)-\mathcal S(z)} - \frac{I(v_0^\circ)}{\mathcal S(v_0^\circ)-\mathcal S(z)} - \frac{I(v_1^\circ)}{\mathcal S(v_1^\circ)-\mathcal S(z)}\right],
    $$
    where $\mu_z$ is a constant set to have $\bar \partial_{\mathcal S} \bar {\mathcal S} = 1$. We define $\partial_{\mathcal S}I := \overline{\bar \partial_{\mathcal S}\bar I}$.

    Similarly, given $K$ locally defined on $\lozenge G$, we define $\bar \partial_{\omega}K$ locally on $G^\bullet \cup G^\circ$ by (recall the definition of primal-facing orientations $\operatorname{n}$ from \eqref{eq:smirnov}):
    $$
    \partial_{\omega}K(v):= \frac{1}{2i}\sum_{z_k \sim v}K(z_k)(\operatorname{n}(c_{k+1})-\operatorname{n(c_k)),}
    $$
    where $z_k \sim v$ enumerates incident quads $z_k$ of $v$ counterclockwise, and $c_{k+1},c_k\in \Upsilon(G)$ are the two medial vertices incident to $z_k$ (i.e. sides of the quadrilateral) which are both incident to $v$, again enumerated counterclockwise and modulo the number of incident quads (see Figure \ref{fig:lattice}). Again, define $\partial_\omega K := \overline{\bar \partial_\omega \bar K}$.

    Then, define the \emph{s-laplacian}, locally defined on $G^\bullet \cup G^\circ$, as $\Delta_{\mathcal S}I := -4\partial_\omega \bar \partial_{\mathcal S}I$.
\end{definition}
Note that these discrete operators may equivalently be thought of as matrices, e.g. $\bar \partial_{{\mathcal S}}$ as a matrix whose columns and rows are respectively labelled by $G^\bullet \cup G^\circ$ and $\lozenge G$, and transposition transforms it to an operator of the opposite type. We have the following useful properties:
\begin{lemma}\label{lem:s-laplacian}
    We have $\partial_{\mathcal S}H^\delta = \frac{1}{4i}(F^\delta)^2$ and $\bar\partial_{\mathcal S}1=\bar\partial_{\mathcal S}\mathcal S=\bar\partial_{\mathcal S}\mathcal Q =0$.

    $\Delta_\mathcal S$ has real coefficients and is symmetric. That is, $\Delta_\mathcal S = -4\partial_\omega \bar \partial_{\mathcal S} = -4\bar \partial_\omega  \partial_{\mathcal S} = \Delta_\mathcal S ^{\mathsf T}$.
    
    Under an isometry of $\mathbb R^{2,1}$, it remains invariant.
\end{lemma}
\begin{proof}
    See \cite[Corollary 3.3, Lemma 3.2, Proposition 3.7]{chelkak} for everything except for invariance. From the same Proposition, we have
    $$
    \Delta_{\mathcal S}I(v) = \pm \sum_{v'\stackrel{s}{\sim} v}a_{vv'}(I(v')-I(v)) + \sum_{w\stackrel{d}{\sim} v}b_{vw}(I(w)-I(v)),
    $$
    where $\pm$ is $+$ if $v\in G^\bullet$ and $-$ if $v\in G^\circ$, and $v'\stackrel{s}{\sim} v$ enumerates adjacent $v' \in G^\bullet \cup G^\circ$ which is of the same type as $v$ (i.e. either both primal or both dual) whereas $w\stackrel{d}{\sim} v$ enumerates incident $w$ of opposite type.
    
    The real coefficients $a_{vv'}, b_{vw} \in \mathbb R$ are explicitly described in the proof of \cite[Lemma 6.1]{chelkak}. Using their notation, note that $\theta_s$ is the weight parameter and is therefore invariant under isometry. Then it is easy to see that the rest is a function of terms of type $\rho_s\rho_{s+1}\sin\phi_{s+1}$; parsing this notation, if $(vw_1)$ and $(vw_2)$ are two medial vertices with some half-angle $\Phi \in [0, \pi]$ between them, this quantity is exactly $|\mathcal S(w_1)-\mathcal S(v)|^{\frac12}|\mathcal S(w_2)-\mathcal S(v)|^{\frac12} \sin{\Phi} \geq 0$. Then, in terms of the inner products on $\mathbb R^2$ and $\mathbb R^{2,1}$, we have
    \begin{align*}
        |\mathcal S(w_1)-\mathcal S(v)|&|\mathcal S(w_2)-\mathcal S(v)| \sin^2{\Phi} = |\mathcal S(w_1)-\mathcal S(v)||\mathcal S(w_2)-\mathcal S(v)| \cdot \frac{1-\cos 2\Phi}{2}\\
        &= \frac{(\mathcal Q(w_1)-\mathcal Q(v))(\mathcal Q(w_2)-\mathcal Q(v)) -\langle \mathcal S(w_1)-\mathcal S(v), \mathcal S(w_2)-\mathcal S(v) \rangle_{\mathbb R^2}}{2}\\
        &=-\frac{\langle (\mathcal S, \mathcal Q)(w_1)-(\mathcal S, \mathcal Q)(v), (\mathcal S, \mathcal Q)(w_2)-(\mathcal S, \mathcal Q)(v) \rangle_{\mathbb R^{2,1}}}{2},
    \end{align*}
    evidently invariant under isometry.
\end{proof}

A crucial ingredient in the proof is the so-called \emph{s-positivity} property, which says that the integrated Ising observables $H^\delta$ have positive s-laplacians \cite[Corollary 3.8]{chelkak}:
\begin{equation}\label{eq:s-positivity}
    \Delta_{\mathcal S}H^\delta = -4\bar \partial_\omega (F^\delta)^2 \geq 0.
\end{equation}

The strategy for \eqref{eq:laplacian-bound} is to carefully transform $\Delta^\Sigma\tilde{H}^\delta(u)=\Delta^\Sigma_u\tilde{H}^\delta(u)$ step-by-step into a quantity with which we are able to use \eqref{eq:s-positivity}, carefully managing the resulting errors. Note the oscillation of $H^\delta$ is bounded by $O(\rho_ud_u^{-1})$ from \eqref{eq:regularity} and \eqref{eq:discrete-integral}, and the left-hand side is invariant under vertical translations, so we will assume $H^\delta=O(\rho_ud_u^{-1})$. For ease of reference, we use the same step numbers as in \cite[Section 4.2]{chelkak} for corresponding steps.

\subsubsection*{Step 0} $\Delta^\Sigma_u\tilde{H}^\delta(u) = \Delta^{\Sigma(u)}_u\tilde H^\delta(u) + O(d_u^{-1})$.
\begin{proof}
    We use the tangent plane laplacian from \eqref{eq:tangent-plane-def}, and need to bound $D^{\Sigma(u)}H^\delta(u)$. Since the second derivatives of $\vartheta$ are bounded and $|\nabla \vartheta|$ is bounded away from $1$ by assumption, it suffices to note
    \begin{align*}
        \nabla_u \tilde H^\delta(u) &= \int_{\Omega^\delta} \nabla_u \varphi^u(w)H^\delta(w) d^2w\\
        &=\int_{\Omega^\delta} (-\nabla_w +O(\delta^\epsilon\rho_u^{-3}))\varphi^u(w)H^\delta(w) d^2w\\
        &=\int_{\Omega^\delta} \varphi^u(w)\nabla_w H^\delta(w)  d^2w+O(d_u^{-1}) =O(d_u^{-1}),
    \end{align*}
    since $\varphi^u(w):= \rho_u^{-2}\cdot\varphi_0(\rho_u^{-1}\cdot(w-u))$, integrating by parts. 
\end{proof}

\subsubsection*{Step 1} $\Delta^{\Sigma(u)}_u\tilde{H}^\delta(u) = \int_{\Omega^\delta} \Delta^{\Sigma(w)}_w \varphi^u(w)H^\delta(w) d^2w+O(\delta^{\epsilon\beta}{d_u^{-2}}) + O(d_u^{-1})$.
\begin{proof}
    We may swap $\Delta^{\Sigma(u)}_u$ by $\Delta^{\Sigma(u)}_w$ as in the corresponding step in \cite[Section 4.2]{chelkak}, getting $O(\delta^{\epsilon(2+\beta)}{\rho_u^{-2}})=O(\delta^{\epsilon\beta}{d_u^{-2}})$ error, where $\beta$ is the exponent from \eqref{eq:regularity}. Then, we have
    $$
    \int_{\Omega^\delta} (\Delta^{\Sigma(w)}_w-\Delta^{\Sigma(u)}_w) \varphi^u(w)H^\delta(w) d^2w = O(d_u^{-1}),
    $$
    since $\Delta^{\Sigma(w)}_w-\Delta^{\Sigma(u)}_w$ has $O(\rho_u)$ coefficients (again by the $C^2$ bound on $\vartheta$) and $\varphi^u(w)$ has $O(\rho_u^{-4})$ second derivatives, using the bound of $H^\delta$ once more.
\end{proof}
\begin{remark}\label{rem:lorentz-covariance}
    We will be using `rotated', i.e. mapped under Lorentz isometry, s-embeddings henceforth, so let us define some notation. For $w\in \Omega$, we consider the isometry which sends the tangent plane $\vartheta^{(w)}$ of $\vartheta$ at $(w, \vartheta(w))$ to flat plane and $\mathcal Q^\delta$ at (any of the) closest vertex to $0$. Define $\vartheta_{[w]}$ and $\mathcal Q^\delta_{[w]}$ as the image of $\vartheta$ and $\mathcal Q^\delta$ under this isometry, and $\bar\partial_{\mathcal S_{[w]}},\bar\partial_{\omega_{[w]}}$, etc. as the operators on the rotated s-embedding. We will frequently use the $C^2$ bound for $\vartheta$ in the form $\vartheta_{[w]} = O(\delta^2)$ at points $\delta$ away from $w$. In particular, we may utilise the assupmtion $\operatorname{Approx}(0,\delta) = \operatorname{FLAT}(\delta)$ at $O(\delta^{1/2})$ distance scales, including single quad estimates \cite[Lemmas 3.9, 3.10]{chelkak}.

    $H^\delta$ is invariant under such re-embedding thanks to its embedding-agnostic definition; $F^\delta$ transforms to some $F_{[w]}^\delta$, which satisfies $(F_{[w]}^\delta)^2 = 4i\cdot \partial_{\mathcal S_{[w]}}H^\delta$ by Lemma \ref{lem:s-laplacian}.
    

\end{remark}
\subsubsection*{Step 2} \begin{align*} \int_{\Omega^\delta} \Delta^{\Sigma(w)}_w \varphi^u(w)H^\delta(w) d^2w=& \sum_{\substack{v\in G^\bullet\cup G^\circ:\\\mathcal{S}^\delta(v)\in B_{\rho_u}(u)}}\frac{(1+|\nabla\vartheta(\mathcal S^\delta(v))|)^{\frac12}}{(1-|\nabla\vartheta(\mathcal S^\delta(v))|)^{\frac12}}(\mathcal Q^\delta-\vartheta\circ\mathcal S^\delta)(v)H^\delta(v)\Delta_{\mathcal S}(\varphi^u\circ \mathcal S^\delta)(v)\\+O((\delta^{1-\gamma}&d_u^{-1}+\delta^\gamma+\delta^{1-3\gamma-\epsilon}d_u^{-1})\rho_u^{-2}) + O(\delta^{4(1-\gamma)}\rho_u^{-4}d_u^{-1}\cdot\delta^{-2(1-\gamma)}\cdot  \rho_u^2).
\end{align*}
\begin{proof}
    Cover $B_{\rho_u}(u)$ by $\delta^{1-\gamma}\times \delta^{1-\gamma}$ squares, with $2\epsilon < \gamma \ll 1$. For such a square $L$ and a fixed point $w_0 \in L$, we have
    \begin{equation}\label{eq:inside-L-laplacian}
    \int_{L} \Delta^{\Sigma(w)}_w \varphi^u(w)H^\delta(w) d^2w = \int_{L} \Delta^{\Sigma(w_0)}_{w} \varphi^u(w)H^\delta(w) d^2w + O(\delta^{4(1-\gamma)}\rho_u^{-4}d_u^{-1}),
    \end{equation}
    since the coefficients of $\Delta^{\Sigma(w)}_w-\Delta^{\Sigma(w_0)}_w$ are $O(\delta^{1-\gamma})$ and $H^\delta$ oscillates $O(\delta^{1-\gamma}d_u^{-1})$. 
    
    Similarly, the right-hand side sum, restricted to $L$, is equal to
    $$
    \sum_{\substack{v\in G^\bullet\cup G^\circ:\\\mathcal{S}^\delta(v)\in L}}\frac{(1+|\nabla\vartheta^{(w_0)}(\mathcal S^\delta(v))|)^{\frac12}}{(1-|\nabla\vartheta^{(w_0)}(\mathcal S^\delta(v))|)^{\frac12}}(\mathcal Q^\delta-\vartheta^{(w_0)}\circ\mathcal S^\delta)(v)H^\delta(v)\Delta_{\mathcal S}(\varphi^u\circ \mathcal S^\delta)(v),
    $$
    up to $O(\delta^{-2\gamma}\cdot\delta^{2-\gamma}\cdot  \delta\cdot\rho_u^{-4}\cdot \delta^{1-\gamma}\cdot d_u^{-1})=O(\delta^{4-4\gamma}\cdot\rho_u^{-4}\cdot d_u^{-1})$ error, thanks to $\operatorname{UNIF}(\delta)$, $\nabla \vartheta^{(w_0)}-\nabla \vartheta = O(\delta^{1-\gamma})$, $\vartheta^{(w_0)}- \vartheta = O(\delta^{2-2\gamma})$, and $\Delta_{\mathcal S}(\varphi^u\circ \mathcal S^\delta)(v) = O(\delta\cdot \rho_u^{-4})$. The last bound holds since the s-Laplacian (which has $O(\delta^{-1})$ coefficients) would kill up to first order terms in the Taylor expansion of $\varphi^u$ by Lemma \ref{lem:s-laplacian}.

    Then, under a Lorentz isometry which sends $\vartheta^{(w_0)}$ to the flat plane $z=0$, \eqref{eq:inside-L-laplacian} (i.e. the product of the area measure and the Laplace-Beltrami operator), as well as $\Delta_{\mathcal{S}}$ and $H^\delta$, remain invariant, while
    $\frac{(1+|\nabla\vartheta^{(w_0)}(\mathcal S^\delta(v))|)^{\frac12}}{(1-|\nabla\vartheta^{(w_0)}(\mathcal S^\delta(v))|)^{\frac12}}(\mathcal Q^\delta-\vartheta^{(w_0)}\circ\mathcal S^\delta) = \mathcal Q^\delta_{[w_0]}$. Then we may appeal to \cite[Proposition 3.12]{chelkak} as in Step 2 of \cite[Section 4.2]{chelkak} and obtain the claim.
\end{proof}

\subsubsection*{Step 3}
\begin{align*}
    &\sum_{\substack{v\in G^\bullet\cup G^\circ:\\\mathcal{S}^\delta(v)\in B_{\rho_u}(u)}}\frac{(1+|\nabla\vartheta(\mathcal S^\delta(v))|)^{\frac12}}{(1-|\nabla\vartheta(\mathcal S^\delta(v))|)^{\frac12}}(\mathcal Q^\delta-\vartheta\circ\mathcal S^\delta)(v)H^\delta(v)\Delta_{\mathcal S}(\varphi^u\circ \mathcal S^\delta)(v)\\
    =&-4\sum_{\substack{z\in \lozenge(G):\\\mathcal{S}^\delta(z)\in B_{\rho_u}(u)}}\partial_{\omega_{[z]}}^{\mathsf T}(\mathcal Q_{[z]}^\delta H^\delta)(z)\bar\partial_{\mathcal S_{[z]}}(\varphi^u\circ \mathcal S^\delta)(z) +O(d_u^{-1}),
\end{align*}
\begin{proof}
    `Summing by parts' and summing over quads instead of vertices gives rise to the transpose in the linear operator as usual; if we used a fixed factorisation $\Delta_{\mathcal S}=-4\partial_{\omega}\bar\partial_{\mathcal S}$ we wouldn't have an error term. Let us carefully study the error term arising from using spatially varying factorisation.
    
    A given quad $z$ appears four times in the original sum, from the s-Laplacians evaluated at each of its four vertex corners $v$. By Lorentz invariance, we may factor $\Delta_{\mathcal S}=-4\partial_{\omega_{[v]}}\bar\partial_{\mathcal S_{[v]}}$. Then, modifying each coefficient of $\partial_{\omega_{[v]}},\bar\partial_{\mathcal S_{[v]}}$ into the corresponding coefficient in $\partial_{\omega_{[z]}},\bar\partial_{\mathcal S_{[z]}}$, we may write
    \begin{align*}
    &\sum_{\substack{v\in G^\bullet\cup G^\circ:\\\mathcal{S}^\delta(v)\in B_{\rho_u}(u)}}\frac{(1+|\nabla\vartheta(\mathcal S^\delta(v))|)^{\frac12}}{(1-|\nabla\vartheta(\mathcal S^\delta(v))|)^{\frac12}}(\mathcal Q^\delta-\vartheta\circ\mathcal S^\delta)(v)H^\delta(v)\Delta_{\mathcal S}(\varphi^u\circ \mathcal S^\delta)(v)\\
    =-4&\sum_{\substack{z\in \lozenge(G):\\\mathcal{S}^\delta(z)\in B_{\rho_u}(u)}}\partial_{\omega_{[z]}}\left[\frac{(1+|\nabla\vartheta\circ\mathcal S^\delta|)^{\frac12}}{(1-|\nabla\vartheta\circ\mathcal S^\delta|)^{\frac12}}(\mathcal Q^\delta-\vartheta\circ\mathcal S^\delta)H^\delta\right](z)\bar\partial_{\mathcal S_{[z]}}(\varphi^u\circ \mathcal S^\delta)(z) +O(d_u^{-1}).
\end{align*}
This bound comes from $(\mathcal Q^\delta-\vartheta\circ\mathcal S^\delta)H^\delta = O(\delta \rho_ud_u^{-1})$ and the fact that, for the purposes of comparing coefficients of $\bar\partial_{\mathcal S_{[z]}}$ and $\bar\partial_{\mathcal S_{[v]}}$ (which both kill constants by Lemma \ref{lem:s-laplacian}), we may use the size $O(\delta \rho_u^{-3})$ oscillation of $\varphi^u\circ \mathcal S^\delta$ at distance scale $O(\delta)$.  Again by the $C^2$ bound on $\vartheta$, perturbing a coefficient of $\bar\partial_{\mathcal S_{[v]}}$ or $\bar\partial_{\omega_{[v]}}$ at this distance scale multiplies it by $O(\delta)$, and the total error is $ O(\rho_u^{2}\delta^{-2}\cdot\delta\cdot\delta \rho_ud_u^{-1}\cdot\delta^{-1}\delta\rho_u^{-3} )=O(d_u^{-1})$.

We also have 
$$
\frac{(1+|\nabla\vartheta^{}(\mathcal S^\delta(v))|)^{\frac12}}{(1-|\nabla\vartheta(\mathcal S^\delta(v))|)^{\frac12}}(\mathcal Q^\delta-\vartheta\circ\mathcal S^\delta) (v)= (\mathcal Q^\delta_{[z]}-\vartheta_{[z]}\circ\mathcal S^\delta_{[z]})(v) + O(\delta^2)=\mathcal Q^\delta_{[z]}(v) + O(\delta^2),
$$
by first transforming the left-hand side exactly to $(\mathcal Q^\delta_{[v]}-\vartheta_{[v]}\circ\mathcal S^\delta_{[v]})(v) = O(\delta)$ then noticing that perturbing it at $O(\delta)$ distance scale multiplies it by $O(\delta)$. Since we have $O(\rho_u^2 \delta^{-2})$ terms, the $O(\delta^2)$ term in the right-hand side again contributes $O(\rho_u^2 \delta^{-2}\cdot\delta^2\rho_ud_u^{-1}\cdot \rho_u^{-3})=O(d_u^{-1})$ in the final right-hand side.
\end{proof}

We now write
$$\partial_{\omega_{[z]}}^{\mathsf T}(\mathcal Q_{[z]}^\delta H^\delta)(z) = \frac{1}{4i}A_{[z]}F_{[z]}^\delta(z)^2+\frac12 B_{[z]}|F_{[z]}^\delta(z)|^2-\frac{1}{4i}C_{[z]}\overline{F_{[z]}^\delta(z)}^2,$$ as in \cite[Section 3.2]{chelkak}, for
$$
A_{[z]} = \partial_{\omega_{[z]}}^{\mathsf T}(\mathcal Q_{[z]}^\delta \mathcal S_{[z]}^\delta)(z),\,B_{[z]} = \partial_{\omega_{[z]}}^{\mathsf T}((\mathcal Q_{[z]}^\delta)^2)(z),\,\text{and}\,C_{[z]} = \partial_{\omega_{[z]}}^{\mathsf T}(\mathcal Q_{[z]}^\delta\bar {\mathcal S}_{[z]}^\delta)(z).
$$
Our goal for Step 4 is to estimate the contribution of the $B_{[z]}$ term.

\subsubsection*{Step 4}
\begin{align*}
    &\sum_{\substack{z\in \lozenge(G):\\\mathcal{S}^\delta(z)\in B_{\rho_u}(u)}}\frac12 B_{[z]}|F_{[z]}^\delta(z)|^2\bar\partial_{\mathcal S_{[z]}}(\varphi^u\circ \mathcal S^\delta)(z) \\ =& O(\delta^{1-\gamma}d_u^{-1}\rho_u^{-1}+\delta^{\beta(1-\gamma)}d_u^{-\beta}\rho_u^{-2}+\delta^{1-\gamma-\epsilon}d_u^{-1}\rho_u^{-2}+\delta^\gamma d_u^{-1}\rho_u^{-1}),
\end{align*}
\begin{proof}
As in Step 4 of \cite[Section 4.2]{chelkak} and Step 2, we again cover $B_{\rho_u}(u)$ with squares of size $\delta^{1-\gamma}\times \delta^{1-\gamma}$. In any square $L$ with $z_0\in L$, we have
    $$\sum_{\substack{z\in \lozenge(G):\\\mathcal{S}^\delta(z)\in L}}\frac12 B_{[z]}|F_{[z]}^\delta(z)|^2\bar\partial_{\mathcal S_{[z]}}(\varphi^u\circ \mathcal S^\delta)(z)=\sum_{\substack{z\in \lozenge(G):\\\mathcal{S}^\delta(z)\in L}}\frac12 B_{[z]}|F_{[z_0]}^\delta(z)|^2\bar\partial_{\mathcal S_{[z_0]}}(\varphi^u\circ \mathcal S^\delta)(z) + O(\delta^{3-3\gamma}\rho_u^{-3}d_u^{-1}),$$
    since $B_{[z]} = O(\delta^2)$ by \cite[Lemma 3.9]{chelkak}, and $|F_{[z]}^\delta(z)|^2  = |\bar\partial_{\mathcal S_{[z]}}H^\delta(z)|^2= O(d_u^{-1})$ by Remark \ref{rem:lorentz-covariance} and \eqref{eq:regularity}, and $\bar\partial_{\mathcal S_{[z]}}(\varphi^u\circ \mathcal S^\delta)(z) = O(\rho_u^{-3})$ as in Step 3. The error in perturbing the operator $\bar\partial_{\mathcal S_{[z]}}$ to $\bar\partial_{\mathcal S_{[z_0]}}$ was obtained by multiplying $O(\delta^{1-\gamma})$ to these typical sizes as usual, resulting in the term $O(\delta^{-2\gamma}\cdot\delta^2\cdot\delta^{1-\gamma}\cdot d_u^{-1} \rho_u^{-3}) = O(\delta^{3-3\gamma}\rho_u^{-3}d_u^{-1})$. Multiplying by the number of squares $\rho_u^2\delta^{-2(1-\gamma)}$ yields the first term in the final bound.

    Then further replacing $|F_{[z_0]}^\delta(z)|^2\bar\partial_{\mathcal S_{[z_0]}}(\varphi^u\circ \mathcal S^\delta)(z)$ by $|F_{[z_0]}^\delta(z_0)|^2\bar\partial_{\mathcal S_{[z_0]}}(\varphi^u\circ \mathcal S^\delta)(z_0)$ results in the second and third terms of the final bound, which are taken from corresponding terms from Step 4 of \cite[Section 4.2]{chelkak}, since we use exactly the same bounds.

    Then it remains to estimate
    \begin{align*}
    &|F_{[z_0]}^\delta(z_0)|^2\bar\partial_{\mathcal S_{[z_0]}}(\varphi^u\circ \mathcal S^\delta)(z_0)\sum_{\substack{z\in \lozenge(G):\\\mathcal{S}^\delta(z)\in L}}\frac12 B_{[z]}\\=&|F_{[z_0]}^\delta(z_0)|^2\bar\partial_{\mathcal S_{[z_0]}}(\varphi^u\circ \mathcal S^\delta)(z_0)\sum_{\substack{z\in \lozenge(G):\\\mathcal{S}^\delta(z)\in L}}\frac12\left[\partial_{\omega_{[z_0]}}^{\mathsf T}((\mathcal Q_{[z_0]}^\delta)^2)(z) + O(\delta^{1-\gamma}\delta^2)\right]\\
    =&|F_{[z_0]}^\delta(z_0)|^2\bar\partial_{\mathcal S_{[z_0]}}(\varphi^u\circ \mathcal S^\delta)(z_0)\cdot O(\delta\cdot\delta^{1-\gamma}+\delta^{1-\gamma}\delta^2\cdot\delta^{-2\gamma})=O(d_u^{-1}\rho_u^{-3}\delta^{2-\gamma}),
    \end{align*}
    since we have $ B_{[z]} = \partial_{\omega_{[z]}}^{\mathsf T}((\mathcal Q_{[z]}^\delta)^2)(z) =  \partial_{\omega_{[z_0]}}^{\mathsf T}((\mathcal Q_{[z_0]}^\delta)^2)(z) + O(\delta^{1-\gamma}\delta^2)$, where we again multiplied $O(\delta^{1-\gamma})$ to the typical term size $\delta^2$ to get the error term. We then used \cite[Proposition 3.11]{chelkak} to deal with the sum of $\frac12\partial_{\omega_{[z_0]}}^{\mathsf T}((\mathcal Q_{[z_0]}^\delta)^2)(z)$. Then it is straightforward to assemble the last term in the final bound.
\end{proof}
Now we study the remaining terms. In fact, the goal is not to trivialise them at this stage:
\subsubsection*{Step 5}
\begin{align*}
\Delta^\Sigma_u\tilde H^\delta(u)=&-4\sum_{\substack{z\in \lozenge(G):\\\mathcal{S}^\delta(z)\in B_{\rho_u}(u)}}\left( A_{[z]}\bar\partial_{\mathcal S_{[z]}}(\varphi^u\circ \mathcal S^\delta)(z)+\overline C_{[z]}\partial_{\mathcal S_{[z]}}(\varphi^u\circ \mathcal S^\delta)(z)\right)\partial_{\mathcal S_{[z]}}H^\delta +O(\ldots),
\end{align*}
where $O(\ldots)$ refers to the errors thus seen. This part is verbatim \cite[Section 4.2]{chelkak}.

\subsubsection*{Step 6}
\begin{align*}
&\Delta^\Sigma_u\tilde H^\delta(u)=-4\sum_{\substack{z\in \lozenge(G):\\\mathcal{S}^\delta(z)\in B_{\rho_u}(u)}}\partial_{\omega_{[z]}}^{\mathsf T}(\mathcal Q_{[z]}^\delta \cdot (\varphi^u\circ\mathcal S^\delta))(z)\partial_{\mathcal S_{[z]}^\delta}H^\delta (z)+O(\delta\rho_u^{-2}\cdot d_u^{-1})+O(\ldots)\\
=&-4\sum_{\substack{z\in \lozenge(G):\\\mathcal{S}^\delta(z)\in B_{\rho_u}(u)}}\partial_{\omega_{[z]}}^{\mathsf T}((\mathcal Q^\delta_{[z]}-\vartheta_{[z]}\circ\mathcal S^\delta_{[z]}) \cdot (\varphi^u\circ\mathcal S^\delta))(z)\partial_{\mathcal S_{[z]}^\delta}H^\delta(z) +O(d_u^{-1})+O(\ldots)\\
=&\sum_{\substack{v\in G^\bullet \cup G^\circ:\\\mathcal{S}^\delta(v)\in B_{\rho_u}(u)}}\frac{(1+|\nabla\vartheta(\mathcal S^\delta(v))|)^{\frac12}}{(1-|\nabla\vartheta(\mathcal S^\delta(v))|)^{\frac12}}\cdot(\mathcal Q^\delta_{[z]}-\vartheta_{[z]}\circ\mathcal S^\delta_{[z]})(v) \cdot (\varphi^u\circ\mathcal S^\delta)(v)\Delta_{\mathcal S^\delta}H^\delta +O(d_u^{-1})+O(\ldots)
\end{align*}
\begin{proof}
The first equality is virtually identical to Step 6 of \cite[Section 4.2]{chelkak}, noting that $O(\rho_u^2\delta^{-2}\cdot \delta^3\rho_u^{-4}\cdot d_u^{-1})=O(\delta\rho_u^{-2}\cdot d_u^{-1})$. Then, inserting $ \vartheta_{[z]} = O(\delta^2)$ costs $ O(\rho_u^2\delta^{-2}\cdot \delta^2\cdot \rho_u^{-2}\cdot d_u^{-1}) = O(d_u^{-1}).$ 

The last equality uses the same idea as Step 3; the error term comes from the $O(\delta)$ added multiplied to the typical size, so we have $O(\delta\cdot \rho_u^2\delta^{-2}\cdot\delta\cdot \rho^{-2}_u\cdot d_u^{-1}) = O(d_u^{-1})$, using $\Delta_{\mathcal S}H^\delta=-4\bar\partial_{\omega}\partial_{\mathcal S}H^\delta = i\bar \partial_\omega F^2 = O(d_u^{-1})$.

\end{proof}

\subsubsection*{Step 7}
$\Delta^\Sigma_u\tilde H^\delta(u)= O(\ldots)$. This part is verbatim \cite[Section 4.2]{chelkak}.

Recall $\beta$ is a small constant already fixed respectively by the regularity estimate \eqref{eq:regularity} and $\eta$ whose maximum is fixed by the RSW argument near the boundary. If we impose $0<2\epsilon<\gamma\ll1$ and $\epsilon <\min(\frac{1}{9},\frac{\beta}{2(1+\beta)})\eta\ll1$, it is a straightforward, albeit lengthy, check to see that there is some $\gamma, \alpha$ such that the accumulated error terms obey \eqref{eq:laplacian-bound} uniformly in $\Omega_{\operatorname{int}(\eta)}^\delta$. The proof is then complete.

\end{document}